\documentclass{birkjour}

\usepackage{graphicx}
\usepackage[ansinew]{inputenc}    
\usepackage{fancyhdr}
\usepackage{easybmat}
\usepackage{color}
\usepackage{verbatim}
\usepackage{amsmath,amstext,amssymb,}
\usepackage{mathrsfs,amscd}

\newtheorem{thm}{Theorem}[section]
\newtheorem{corollary}[thm]{Corollary}
\newtheorem{lemma}[thm]{Lemma}
\newtheorem{proposition}[thm]{Proposition}

\newtheorem{example}[thm]{Example}

\newcommand{\R}{{\mathbb{R}}}

\begin{document}

\title[Envelope of mid-planes of a surface]{{ Envelope of mid-planes of a surface and some classical notions of affine differential geometry}}

\author[A.Cambraia Jr. ]{Ady Cambraia Jr.}
\address{
Departamento de Matem\'atica, UFV, Vi\c cosa, Minas Gerais, Brazil}

\email{ady.cambraia@ufv.br}
\author[M. Craizer]{Marcos Craizer}
\address{Departamento de Matem\'atica-PUC-Rio, Rio de Janeiro, Brazil.}
\email{craizer@puc-rio.br}

\begin{abstract}
For a pair of points in a smooth locally convex surface in $3$-space, its mid-plane is the plane containing its mid-point and the intersection line of the corresponding pair of tangent planes. In this paper we show that the limit of mid-planes when one point tends to the other along a direction is the Transon plane of the direction. Moreover, the limit of the envelope of mid-planes is non-empty for at most six directions, and, in this case, it coincides with the center of the Moutard's quadric. These results establish a connection between these classical notions of affine differential geometry and the apparently unrelated concept of envelope of mid-planes of a surface.
We call the limit of envelope of mid-planes the {\it affine mid-planes evolute} and prove that, under some generic conditions, it is a regular surface in $3$-space.
\end{abstract}


\thanks{The second author wants to thank CNPq for financial support during the preparation of this paper}


\date{February 14, 2017}

\maketitle

\section{Introduction}
\label{intro}

For a pair of points in a smooth convex planar curve, its mid-line is the line that passes through its mid-point and the intersection of the corres\-pon\-ding tangent lines. The envelope of the $2$-parameter family of mid-lines is an important affine invariant symmetry set associated with the curve. It has applications in computer graphics and has been studied by many authors
(\cite{Sapiro},\cite{Giblin},\cite{Giblin2},\cite{Holtom}). It is well-known that the limit of the mid-lines when one point tends to the other is the affine normal and
 the limit of the envelope of mid-lines is the affine evolute of the curve.
 
The mid-plane is a natural generalization of the concept of mid-line to a smooth locally convex surface $S$: For $p_1,p_2\in S$, the mid-plane is the plane that passes through the mid-point $M$ of $p_1$ and $p_2$ and the intersection line of the tangent planes at $p_1$ and $p_2$. We verify that if we fix a tangent vector $T$ and make $p_2\to p_1$ along this direction, the mid-plane converges to the Transon plane of the surface at $p_1$ in the direction $T$ (\cite{Transon}). 

What can we say about the envelope of the mid-planes of a surface $S$ in $3$-space? For a general description of this envelope, see \cite{Cambraia}. In this paper we are interested in understanding this set when
$p_2\to p_1$. It turns out that, when $p_2\to p_1$ along the direction $T$, the limit of this envelope consists of the solution of a system
of $4$ equations. The first equation of this system defines the Transon plane of $T$ at $p_1$, the first and second define the line of the cone of B.Su of $T$ at $p_1$, and the first, second and third equations define the center of the Moutard quadric of $T$ at $p_1$. The notions of Transon Plane, cone of B.Su and Moutard's quadric are classical, but there are some modern references (\cite{Juttler},\cite{Buchin}). The close connection between these classical notions of affine differential geometry and the apparently unrelated concept of envelope of mid-planes is a main point of this paper.

But what happens with the fourth equation? We verify that the tangent directions that leads to some solution of this system of $4$ equations are zeros of a polynomial equation of degree $6$, thus they are at most $6$. The set of solutions of this system is an affine invariant set that, up to our knowledge, has not yet been considered. We call it the {\it  Affine Mid-Planes Evolute}. We verify that under some generic conditions, the affine mid-planes evolute is a regular surface in $3$-space.

This work is part of the doctoral thesis of the first author under the supervision of the second author.

\section{Some classical concepts in affine differential geometry} 

In this section, we recall the notions of Transon plane, B.Su's cone and Moutard's quadric of a given direction at a point of a smooth locally convex surface.

\subsection{Preliminaries}

Consider a smooth convex surface $S$ and a point $p_0\in S$. Assume that  $p_0=(0,0,0)$ and that the tangent plane at $p_0$ is $z=0$. Assume also that the axes $x$ and $y$ are orthogonal in the Blaschke metric of $S$. Then, close to $p_0$, $S$ is the graph of a function $f$ that can be written as
\begin{equation}\label{eq:GraphS}
f(x,y)=\frac{1}{2}(x^2+y^2)+f_3(x,y)+ f_4(x,y) +O(5), %
\end{equation}
where 
\begin{equation}\label{eq:GraphS1}
f_3(x,y)=\sum_{i=0}^3f_{3-i,i}x^{3-i}y^{i},\ \  f_4(x,y)=\sum_{i=0}^{4}f_{4-i,i}x^{4-i}y^i 
\end{equation}
are homogeneous polynomials of degree $k$, $k=3,4$, and $O(n)$ denotes some expression of order $\geq n$ in $(x,y)$.

We may also assume that the affine normal vector of $S$ at the origin is $(0,0,1)$. In this case, the apolarity condition implies
that $3f_{3,0}+f_{1,2}=0$ and $3f_{0,3}+f_{2,1}=0$ (see \cite{Nomizu}). Thus we can write 
\begin{equation}\label{eq:CubicaNormalizada}
f_3(x,y)=a(x^3-3xy^2)+b(y^3-3yx^2), 
\end{equation}
where $6a=C(e_1,e_1,e_1)$, $6b=C(e_2,e_2,e_2)$, $e_1=(1,0,0)$, $e_2=(0,1,0)$ and $C$ denotes the cubic form at the origin.

\subsection{Transon planes, cone of B.Su and Moutard's quadric}\label{sec:MedialCurve}

We begin with a two-hundred years old result of A.Transon (\cite{Transon}). For a mo\-dern reference, see \cite{Juttler}. 
For the sake of completeness, we give a proof below.

\begin{proposition}\label{prop:Transon}
Consider a smooth convex surface $S\subset\mathbb{R}^3$, $p_0\in S$ and $T\in T_{p_0}S$. Then the affine normal lines of the planar curves obtained as the intersection of $S$ with planes containing $T$ form a plane, which is called the {\it Transon plane} of the tangent $T$ at $p_0$.
\end{proposition}

\begin{proof}
We may assume that $S$ is the graph of $f$ given by equation \eqref{eq:GraphS}, $p_0=(0,0,0)$ and $T=(1,0,0)$. Take then a plane of the form $y=\lambda z$, $\lambda\in\mathbb{R}$. Using formulas \eqref{eq:GraphS} and \eqref{eq:GraphS1}, the projection of the corresponding section of $S$
on the $xz$ plane is given by
\begin{equation}\label{eq:proj}
z=\frac{1}{2}x^2+f_{3,0}x^3+\left( \frac{\lambda^2}{8}+\frac{\lambda}{2}f_{2,1}+f_{4,0} \right) x^4+ O(5).
\end{equation}
From appendix \ref{app1}, the affine normal direction of this projection is $(-2f_{3,0},1)$. Thus the affine normal of the section is contained in the 
plane $x+2f_{3,0}z=0$, which is independent of $\lambda$.
\end{proof}

Similar calculations show that when $S$ is the graph of a function given by \eqref{eq:GraphS}, the Transon plane at the origin in a direction $T=(\xi,\eta,0)$  is given by $G(\xi,\eta,X)=0$, where $X=(x,y,z)$ and 
\begin{equation}\label{planodetranson}
G(\xi,\eta,X)=\frac{\xi}{2}(\xi^2+\eta^2)x+\frac{\eta}{2}(\xi^2+\eta^2)y+f_3(\xi,\eta)z,
\end{equation}
(see also \cite{Juttler}). When the tangent plane is $z=0$ and $T$ is a tangent vector, we shall write simply $T=(\xi,\eta)$ instead of $T=(\xi,\eta,0)$.

\begin{figure}[ht]
  \centering
  \includegraphics[scale=0.65]{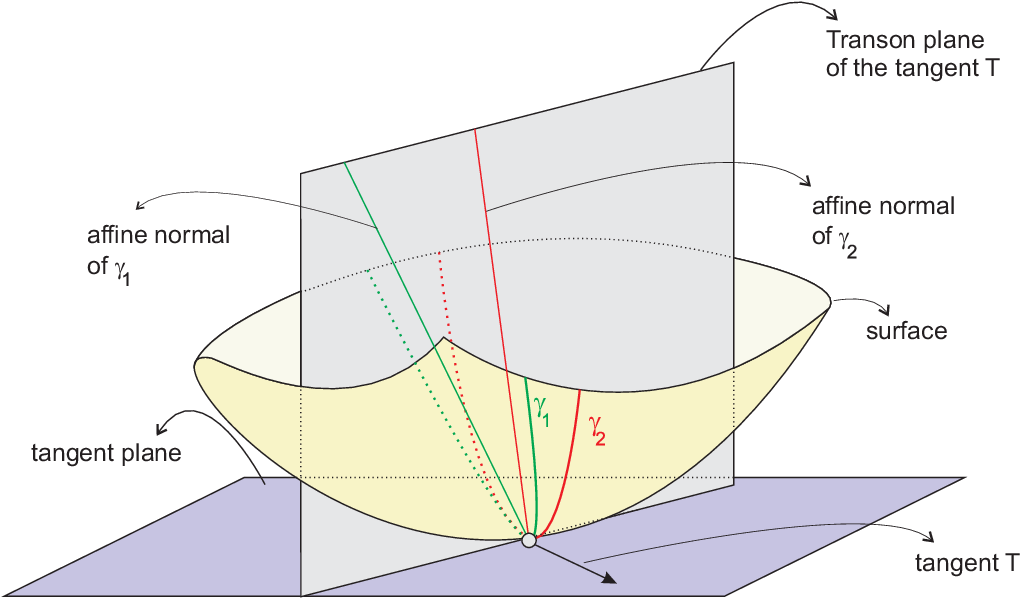}\\
  \caption{The Transon plane of the tangent $T$: Formed by the affine normal lines of the planar sections that contain $T$.}
\end{figure}

Consider now the family of all Transon planes obtained as the direction of $T$ varies. The envelope of this family is called {\it cone of B.Su} and is
obtained by solving the equations $G=G_{\xi}=G_{\eta}=0$, where
\begin{equation}
\left\{
\begin{array}{c}
G_{\xi}=\frac{1}{2}(3\xi^2+\eta^2)x+(\xi\eta) y+\left((f_3)_{\xi}\right)z\\
G_{\eta}=(\xi\eta) x+\frac{1}{2}(\xi^2+3\eta^2)y+\left((f_3)_{\eta}\right)z.
\end{array}
\right.
\end{equation}
Since $G$ is homogeneous of degree $3$, Euler's relation says that 
$$3G=\eta G_{\eta}+\xi G_{\xi}.$$
Thus the cone of B.Su is obtained by solving the equations
$G_{\xi}=G_{\eta}=0$. We shall denote a vector in the direction of this line by $s(T)=s(p,T)$.
We can thus calculate $s(T)$ as a multiple of the vector product of the normal vectors of $G_{\xi}=0$ and $G_{\eta}=0$. 
In particular, for $(\xi,\eta)=(1,0)$ we obtain
\begin{equation}\label{eq:ConeSu10}
s(1,0)=\left( -2f_{3,0}, -2f_{2,1}, 1\right).
\end{equation}

We shall consider also the osculating conics of all planar sections obtained from planes containing $T$. The following proposition is an old result of T.Moutard (\cite{Juttler},\cite{Moutard}). We give a proof for the sake of completeness.  

\begin{proposition}
The union of the osculating conics of all planar sections containing $T$ form a quadric, which is called the {\it Moutard's quadric} of the tangent $T$. 
\end{proposition}
\begin{proof}
Assume that $S$ is the graph of $f$ given by equation \eqref{eq:GraphS}, $p_0=(0,0,0)$ and $T=(1,0)$. Then the projection of the section of $S$ by the plane $y=\lambda z$
is given by equation \eqref{eq:proj}. By appendix \ref{app1}, the affine curvature is given by
$$
\mu=\lambda^2+4\lambda f_{2,1}+8f_{4,0}-20f_{3,0}^2.
$$
The projection of the osculating conic of this section in the plane $xz$ is given by
$$
(x+2f_{3,0}z)^2+\mu z^2-2z=0.
$$
Substituting $\lambda=y/z$ in the above equation, after some calculations we obtain that the osculating conic is contained in 
\begin{equation}
z=\frac{1}{2}(x^2+y^2)+2f_{2,1}yz+2f_{3,0}xz+4(f_{4,0}-2f_{3,0}^2)z^2,
\end{equation}
thus proving the lemma. For later reference, we remark that the center of this Moutard quadric is 
\begin{equation}\label{eq:CenterMoutard}
X=\dfrac{1}{4(2f_{4,0}-5f_{3,0}^2-f_{2,1}^2)} \left(-2f_{3,0},-2f_{2,1},1 \right).
\end{equation}
Notice that if $2f_{4,0}-5f_{3,0}^2-f_{2,1}^2=0$, the center of the Moutard quadric is at infinity. 
\end{proof}



\begin{proposition}\label{prop:centro_quadrica_moutard}
The center of the Moutard's quadric of a tangent $T$ coincides with the center of affine curvature of the intersection $\gamma$ of the plane
generated by $T$ and $s(T)$ with the surface. 
\end{proposition}
\begin{proof}
We may assume that $S$ is the graph of $f$ given by equation \eqref{eq:GraphS}, $p_0=(0,0,0)$ and $T=(1,0)$. 
The plane generated by $T$ and $s(T)$ has equation $y+2f_{2,1}z=0$. Thus, it follows from equation \eqref{eq:proj} that the 
projection of $\gamma$ in the $xz$ plane is given by
$$
z= \dfrac{1}{2}x^2+f_{3,0}x^3+\left(f_{4,0}-\dfrac{f_{2,1}^2}{2}\right)x^4+O(5).
$$
From appendix \ref{app1}, the affine curvature is given 
$\mu=-4(5f_{3,0}^2-2f_{4,0}+f_{2,1}^2)$.
Thus the affine center of curvature at the origin is given by
$$
\dfrac{1}{4(2f_{4,0}-5f_{3,0}^2-f_{2,1}^2)} \left(-2f_{3,0},-2f_{2,1},1 \right),
$$
which coincides with formula \eqref{eq:CenterMoutard}.
\end{proof}

\section{Mid-Planes}

Let $S$ be a regular surface and $p_1,p_2\in S$. Denote by $M(p_1,p_2)$ the mid-point and by $C(p_1,p_2)$ the mid-chord of $p_1$ and $p_2$, i.e., 
$$
M(p_1,p_2)=\dfrac{p_1+p_2}{2},\ \ C(p_1,p_2)=\dfrac{p_1-p_2}{2}.
$$ 
The mid-plane of $(p_1,p_2)$ is the plane that contains $M(p_1,p_2)$ and the line of intersection of the tangent planes at $p_1$ and $p_2$.
Consider $F:S\times S\times\mathbb{R}^3\to\mathbb{R}$  given by
\begin{equation}\label{eq:Normais}
F(p_1,p_2,X)=\left[ (N_1\cdot C)N_2+(N_2\cdot C)N_1\right]\cdot (X-M),
\end{equation}
where $X=(x,y,z)$, $\cdot$ denotes the canonical inner product and $N_i$ is any vector field orthogonal to $T_{p_i}S$, $i=1,2$.
It is not difficult to verify that  the equation of the mid-plane of $(p_1,p_2)$ is $F=0$. 

In this section, we shall assume that $S$ is the graph of a function $f$ defined by equations \eqref{eq:GraphS} and \eqref{eq:GraphS1}, and take
$$
N_i=(-f_x,-f_{y},1)(u_i,v_i)
$$
as normal vectors, where $p_i=(u_i,v_i,f(u_i,v_i))$, $i=1,2$. Write $\Delta u=u_1-u_2$, $\Delta v=v_1-v_2$. 

\subsection{Limit of mid-planes} 

Denote by $O(n)$ terms of order $n$ in $(u_1,u_2,v_1,v_2)$. Next lemma is the main tool to establish the connection between the envelope of mid-planes and the concepts
of Transon planes and B.Su's cone:

\begin{lemma}\label{Lemma:Main3}
We have that
\begin{equation}\label{eq:Main3}
F(u_1,v_1,u_2,v_2,X)=G(\Delta u,\Delta v,X)+O(4).
\end{equation}
where $G$ is given by equation \eqref{planodetranson}. 
\end{lemma}

We can prove this lemma through long but straightforward calculations, the reader can find it in appendix \ref{app2}. 
This lemma has some important consequences, one of them is the following corollary:

\begin{corollary}\label{transon}
Take $p_1=(u_1,v_1,f(u_1,v_1))$ and $p_2=(u_2,v_2,f(u_2,v_2))$ such that $(\Delta u,\Delta v)=t(\xi,\eta)$, $t\in\mathbb{R}$. Then 
the limit of the mid-planes of $(p_1,p_2)$ when $t$ goes to $0$ is the Transon plane of $(\xi,\eta)$ at $p_1$.
\end{corollary}

\begin{proof}
It follows from Lemma \ref{Lemma:Main3} that
$t^3 G(\xi,\eta,X)=O(t^4)$.
When $t\to 0$, this equation converges to $G(\xi,\eta,X)=0$, which is the equation of the Transon plane.
\end{proof}

\subsection{Limit of envelope of mid-planes-I}

We have to consider the system
$$
\left\{
\begin{array}{ll}
F(u_1,v_1,u_2,v_2,X)=0 \\
F_{u_1}(u_1,v_1,u_2,v_2,X)=0 \\
F_{v_1}(u_1,v_1,u_2,v_2,X)=0 \\
F_{u_2}(u_1,v_1,u_2,v_2,X)=0 \\
F_{v_2}(u_1,v_1,u_2,v_2,X)=0 \\
\end{array}
\right.,
$$
which is equivalent to
\begin{equation}\label{sistema}
\left\{
\begin{array}{ll}
F(u_1,v_1,u_2,v_2,X)=0 \\
F_{u_1}(u_1,v_1,u_2,v_2,X)-F_{u_2}(u_1,v_1,u_2,v_2,X)=0 \\
F_{v_1}(u_1,v_1,u_2,v_2,X)-F_{v_2}(u_1,v_1,u_2,v_2,X)=0 \\
F_{u_1}(u_1,v_1,u_2,v_2,X)+F_{u_2}(u_1,v_1,u_2,v_2,X)=0 \\
F_{v_1}(u_1,v_1,u_2,v_2,X)+F_{v_2}(u_1,v_1,u_2,v_2,X)=0 \\
\end{array}
\right. \ .
\end{equation}

\begin{corollary}
We have that
\begin{equation}
\left\{
\begin{array}{ll}
F_{u_1}(u_1,v_1,u_2,v_2,X)-F_{u_2}(u_1,v_1,u_2,v_2,X)=2G_{\xi}(\Delta u,\Delta v)+O(3),\\
F_{v_1}(u_1,v_1,u_2,v_2,X)-F_{v_2}(u_1,v_1,u_2,v_2,X)=2G_{\eta}(\Delta u,\Delta v)+O(3)
\end{array}
\right. \ .
\end{equation}
\end{corollary}
\begin{proof}
Differentiate formula \eqref{eq:Main3}.
\end{proof}

\begin{corollary}\label{Su}
Take $p_1=(u_1,v_1,f(u_1,v_1))$ and $p_2=(u_2,v_2,f(u_2,v_2))$ such that $(\Delta u,\Delta v)=t(\xi,\eta)$. Then 
the limit of the second and third equations of system \eqref{sistema} when $t$ goes to $0$ are $G_{\xi}=G_{\eta}=0$. As a consequence,
$X$ satisfies the first, second and third equations of the limit of system \eqref{sistema} if and only if $X$ belongs to the cone of B.Su at the origin
in the direction $(\xi,\eta)$.
\end{corollary}

\begin{proof}
The second equation of system \eqref{sistema} can be written as $G_{\xi}(\Delta u,\Delta v)=O(3)$. Taking pairs $(p_1,p_2)$ 
such that $(\Delta u,\Delta v)=t(\xi,\eta)$, the second equation of system \eqref{sistema} becomes $t^2G_{\xi}(\xi,\eta,X)=O(t^3)$. 
When $t\to 0$, this equation converges to $G_{\xi}(\xi,\eta,X)=0$. The same reasoning applied to the third equation of system \eqref{sistema} leads to 
the equation $G_{\eta}(\xi,\eta,X)=0$. 
\end{proof}

\subsection{Limit of envelope of mid-planes-II}

Consider $a$ and $b$ given by equation \eqref{eq:CubicaNormalizada} and define 
$$
H_1(\xi,\eta,X)=H_{11}x+H_{12}y+H_{13}z-H_{14},\ $$$$H_2(\xi,\eta,X)=H_{21}x+H_{22}y+H_{23}z-H_{24},
$$
where
$$
H_{11} = \dfrac{1}{2}a\left(5\xi^3+3\eta^2\xi\right)-b\left(3\xi^2\eta+2\eta^3\right), \ \
H_{12} =-3a\eta^3 -\dfrac{3}{2}b\left(\xi^3+3\eta^2\xi\right), 
$$
$$
H_{13} = 2f_{4,0}\xi^3 +\dfrac{3}{2}f_{3,1}\eta\xi^2+f_{2,2}\eta^2\xi+\dfrac{1}{2}f_{1,3}\eta^3, \ \ 
H_{14} = \dfrac{1}{4}(\xi^2+\eta^2)\xi,
$$
$$
H_{21} = -\dfrac{3}{2}a\left(3\xi^2\eta+\eta^3\right)-3b\xi^3, \ \
H_{22} =  -a\left(2\xi^3+3\xi\eta^2\right)+\dfrac{1}{2}b\left(3\xi^2\eta+5\eta^3\right), 
$$
$$
H_{23}  =  \dfrac{1}{2}f_{3,1}\xi^3+f_{2,2}\xi^2\eta+\dfrac{3}{2}f_{1,3}\xi\eta^2+2f_{0,4}\eta^3, \ \
H_{24}  =  \dfrac{1}{4}(\xi^2+\eta^2)\eta. 
$$

Next lemma is the main tool to connect the concept of Moutard's quadric and the envelope of mid-planes:

\begin{lemma}\label{Lemma:Main4}
We have that $F(u_1,v_1,u_2,v_2,X)$ is given by
\begin{equation}\label{eq:Main4}
F=G(\Delta u,\Delta v,X)+H_1(\Delta u,\Delta v,X) \sigma u+H_2(\Delta u,\Delta v,X) \sigma v +O(5),
\end{equation}
where $\sigma u=u_1+u_2$ and $\sigma v=v_1+v_2$.
\end{lemma}

The proof of this lemma is a long but straightforward calculation, the reader can find it in appendix \ref{app2}. 

\begin{corollary}
We have that
\begin{equation}
\left\{
\begin{array}{ll}
F_{u_1}(u_1,v_1,u_2,v_2,X)+F_{u_2}(u_1,v_1,u_2,v_2,X)=2H_1(\Delta u,\Delta v)+O(4),\\
F_{v_1}(u_1,v_1,u_2,v_2,X)+F_{v_2}(u_1,v_1,u_2,v_2,X)=2H_2(\Delta u,\Delta v)+O(4)
\end{array}
\right. \ .
\end{equation}
\end{corollary}
\begin{proof}
Differentiate equation \eqref{eq:Main4}.
\end{proof}

\begin{corollary}
The limit of the system \eqref{sistema} defining the envelope of mid-planes is the system 
\begin{equation}\label{eq:SystemEvolute}
G=G_{\xi}=G_{\eta}=H_1=H_2=0.
\end{equation}
\end{corollary}
\begin{proof}
Take  pairs $(p_1,p_2)$ such that $(\Delta u,\Delta v)=t(\xi,\eta)$.
Then the limit of the fourth and fifth equations when $t$ goes to zero are $H_1(\xi,\eta,X)=0$
and $H_2(\xi,\eta,X)=0$, respectively.
\end{proof}

\begin{proposition}\label{moutard}
If system \eqref{eq:SystemEvolute} admits a solution, then this solution is the center of the Moutard's quadric of $T=(\xi,\eta)$.
\end{proposition}

\begin{proof}
We may assume that the tangent $(\xi,\eta)$ is $(1,0)$.
Taking $\xi=1, \eta=0$ in the system $G=G_{\eta}=H_1=0$ we get
$$
\left\{
\begin{array}{ccc}
  \dfrac{1}{2}x+az & = & 0 \\
  \dfrac{1}{2}y-3bz & = & 0 \\
    \dfrac{5}{2}ax-\dfrac{3}{2}by+2f_{4,0}z & = & \dfrac{1}{4} \\
\end{array}
\right. \ .
$$
The solution of this system is
$$
(x,y,z)= \dfrac{1}{4(2f_{4,0}-5a^2-9b^2)} \left( -2a, 6b, 1 \right),
$$
which is exactly the center of the Moutard's quadric of $T=(1,0)$ at the origin (see formula \eqref{eq:CenterMoutard}).
\end{proof}

\section{Affine Mid-Planes Evolute}

Any direction in the tangent plane of $S$ at $p$ can be represented by a vector $T$ with $h(T)=1$, where $h$
denotes the Blaschke metric. Thus we can write $(p,T)\in T^1S$, where $T^1S$ denotes the unit tangent bundle of $S$.

Since $G=G_{\xi}=G_{\eta}=0$ are linearly dependent, we shall discard the equation $G=0$.
Denote by $D:T^1S\longrightarrow \mathbb{R}$ the determinant of the extended matrix of 
$$
G_{\xi}=G_{\eta}=H_1=H_2=0.
$$
Then the system \eqref{eq:SystemEvolute} admits a solution if and only if $D(p,T)=0$.

\begin{lemma}\label{lemma:q}
For each $p\in S$, the system \eqref{eq:SystemEvolute} admit solutions if and only if $T=(\xi,\eta)$ is a root of a homogeneous polynomial of degree $6$.
 \end{lemma}

\begin{proof}
We may assume $p=(0,0,0)$, $T=(\xi,\eta)$ and $S$ is the graph of $f$ given by equation \eqref{eq:CubicaNormalizada}. Straightforward calculations show that  
$$
D(p,T) = \frac{3}{32}(\xi^2+\eta^2)^2 \cdot q(\xi,\eta),
$$
where $q(\xi,\eta)=12q_3(\xi,\eta)+q_4(\xi,\eta)$ is a homogeneous polynomial of degree $6$,  $q_3$ is given by
$$
ab\xi^6+3(a^2-b^2)\xi^5\eta-15ab\xi^4\eta^2+10(b^2-a^2)\xi^3\eta^3+15ab\xi^2\eta^4+3(a^2-b^2)\xi\eta^5-ab\eta^6
$$
and
$$
q_4=-f_{3,1}\xi^6+(4f_{4,0}-2f_{2,2})\xi^5\eta+(2f_{3,1}-3f_{1,3})\xi^4\eta^2+4(f_{4,0}-f_{0,4})\xi^3\eta^3
$$
$$
+(3f_{3,1}-2f_{1,3})\xi^2\eta^4+(2f_{2,2}-4f_{0,4})\xi\eta^5+f_{1,3}\eta^6.
$$
Thus $D=0$ admits at most $6$ solutions.
\end{proof}

\begin{corollary}
For each $p\in S$, the system \eqref{eq:SystemEvolute} admits a solution for at most $6$ values of the direction $T$.
\end{corollary}

\begin{example}
Consider a point $p\in S$ with $f_4=0$. We may assume, by a rotation of the tangent plane, that $b=0$. Then
$$
q_3=3a^2\xi^5\eta-10a^2\xi^3\eta^3+3a^2\xi\eta^5.
$$
This polynomial has exactly six roots, $(\xi,\eta)=(\cos\frac{k\pi}{6},\sin\frac{k\pi}{6})$, $0\leq k\leq 5$. 
\end{example}

For each $p\in S$ and $T_i(p)$, $1\leq i\leq 6$, given by the above corollary, define $X_i(p)$ as the solution of the system \eqref{eq:SystemEvolute} and write $E(p)=\cup_{i=1}^6 X_i(p)$. Up to our knowledge, the set $E(p)$ has not yet been considered. We call the set $E(p)$, $p\in S$, the {\it Affine Mid-Planes Evolute} of $S$.

\section{Regularity of the Affine Mid-Planes Evolute}

In this section we study the regularity of the branches of the Affine Mid-Planes Evolute. We begin by showing that, under certain conditions, the vector fields $T_i(p)$, $1\leq i\leq 6$, and the corresponding map $X_i(p)$ are smooth. 

\begin{lemma}
Let $p_0\in S$ and assume that there exists $(\xi_0,\eta_0)\in T_{p_0}S$ which is simple root of $q(p_0)$, where $q$ is defined in Lemma \ref{lemma:q}. 
Then there exists a neighborhood $U$ of $p_0$ and a map $\left(Id,(\xi,\eta)\right):U\to T^1S$ such that $(\xi,\eta)(p_0)=(\xi_0,\eta_0)$ and  $(\xi,\eta)(p)$
is a simple root of $q(p)$. Denoting by $X(p)$ the center of the Moutard quadric of $(\xi,\eta)(p)$, the map $X:U\to\mathbb{R}^3$ is differentiable.
\end{lemma}
\begin{proof}
The first claim follows from the implicit function theorem and the second one from the formula of the center of the Moutard quadric. 
\end{proof}

From now on, we shall assume that, for a given branch of the Affine Mid-Planes Evolute, $T=T(p)$ and $X=X(p)$ are smooth functions of $p\in U$. 
We now look for conditions under which this branch of the  Affine Mid-Planes Evolute is a regular surface.

Assume that the frame $\{T_1(p),T_2(p)\}$ is $h$-orthonormal and that the cubic form vanishes at $T_1(p)$. Then $S\times S$
is locally parameterized by
$$
(p,x,y)\to \left( p, (xT_1(p)+yT_2(p)+f(p,x,y)\zeta(p)) \right),
$$
where $\zeta(p)$ is the affine normal vector at $p$. We may write
$$
f(p,x,y)=\frac{1}{2}(x^2+y^2)+b(p)(y^3-3yx^2)+O(4)(p)(x,y),
$$
where $b^2(p)$ is the Pick invariant of $S$ at $p$. Consider $\mathbf{G}:T^1S\to \mathbb{R}^3$ given by 
$$
\mathbf{G}(p,\xi,\eta)=\left(  \frac{\xi}{2}(\xi^2+\eta^2),\frac{\eta}{2}(\xi^2+\eta^2),b(p)(\eta^3-3\eta\xi^2) \right).
$$
Then, for fixed $(p,\xi,\eta)$, denoting $X=(x,y,z)$, $\mathbf{G}(p,\xi,\eta)\cdot X=0$ is the equation of the Transon plane at $p$ in the direction $(\xi,\eta)$.

Consider a curve $p=p(t)$, $p(0)=p_0$, along the surface $S$ and denote by $T(t)=(\xi(t),\eta(t))$ and $X(t)$ the corresponding values of $T$ and $X$ along the branch. 

\begin{lemma}\label{lemma:Transon1}
Assume that  the cubic form at $p$ does not vanish at $(\xi,\eta)(0)$. 
Then $X'(0)$ belongs to the Transon plane if and only if $b'(0)=0$.
\end{lemma}
\begin{proof}
Write $\mathbf{G}(t)=\mathbf{G}(p(t),\xi(t),\eta(t))$. Differentiating $\mathbf{G}(t)\cdot X=0$ at $t=0$ we obtain 
$$
\mathbf{G}_{\xi}\cdot X\ \xi'+\mathbf{G}_{\eta}\cdot X\ \eta'+\mathbf{G}\cdot X'+b'(\eta^3-3\eta\xi^2)z=0.
$$
For $X=(x,y,z)$ center of Moutard's quadric, $\mathbf{G}_{\xi}\cdot X=\mathbf{G}_{\eta}\cdot X=0$. 
Since the cubic form at $p$ does not vanish at $(\xi,\eta)(0)$, we conclude that $\mathbf{G}\cdot X'=0$ if and only 
if $b'(0)=0$, thus proving the lemma.
\end{proof}

It is not easy to find an explicit example of a branch of the affine mid-planes evolute. The main problem is to find, at each point of the surface, a root of the degree six polynomial $q$. Next example is a very particular case where
the calculations can be easily done.

\begin{example}
Consider the homogeneous surface
$$
S=\{ (u,v,w)\in\R^3| \ uvw=1, u>0, v>0, w>0\},
$$
which is also an affine sphere (\cite[p.97]{Nomizu}). 
Let $G\subset SL(3,\R)$ be the subgroup of matrices $g=g(u,v)$, $u\in\R^{+}$, $v\in\R^{+}$, where  
$$
g(u,v)=
\left[
\begin{array}{ccc}
u & 0 & 0\\
0 & v & 0\\
0 & 0 & \frac{1}{uv}
\end{array}
\right]
$$
The $G$ preserves $S$ and every point of $S$ can be write as $g(u,v)\cdot p_0$, where $p_0=(1,1,1)$. Thus we can calculate a root of 
$q(\eta,\xi)$ and the corresponding center of the Moutard quadric at $p_0$ and then apply $g(u,v)$ to obtain the same data for the other points.
At $p_0=(1,1,1)$, consider the tangent basis 
$$
T_1=\sqrt{2}
\left[
\begin{array}{c}
1/2 \\
1/2 \\
-1
\end{array}
\right] ; \ \ 
T_2=\sqrt{6}
\left[
\begin{array}{c}
1/2 \\
-1/2 \\
0
\end{array}
\right].
$$
From the calculations of \cite[p.97]{Nomizu}, one easily verifies that $h(T_1,T_1)=h(T_2,T_2)=1$, $h(T_1,T_2)=0$, i.e.,
$\{T_1,T_2\}$ is an $h$-orthonormal basis of $T_{p_0}S$. The affine normal at $p_0$ is $\zeta=p_0=(1,1,1)$. 

Consider the  affine change of variables
$$
(u,v,w)=p_0+sT_1+tT_2+r\zeta.
$$
In the $(s,t,r)$ coordinates, $p_0$ corresponds to $(0,0,0)$, $\{(1,0,0),(0,1,0)\}$ 
is an $h$-orthonormal basis of $T_{p_0}S$ and $(0,0,1)$ is the affine normal. 
If we write the $uvw=1$ in the $(s,t,r)$ variables and solve for $r$ we obtain 
$$
r=\frac{1}{2}(s^2+t^2) +\frac{\sqrt{2}}{6}\cdot(s^3-3st^2)+ O(5).
$$
From example 4.3, we obtain that $(\xi,\eta)=(1,0)$ is a root of $q$. Then we can use Proposition 3.8
to obtain the center of the Moutard quadric of $(\xi,\eta)=(1,0)$,
$(s,t,r)=\left(\frac{3\sqrt{2}}{10}, 0, -\frac{9}{10}\right)$,
which corresponds to 
$$
(x_0,y_0,z_0)=\left( \frac{2}{5}, \frac{2}{5}, -\frac{1}{2} \right).
$$
For a general point $(u,v,w)\in S$, the tangent vector
$g(u,v)\cdot T_1$ is a root of $q$, and the corresponding point of the Moutard quadric is 
$$
g(u,v)\cdot (x_0,y_0,z_0)=\left( \frac{2u}{5}, \frac{2v}{5}, -\frac{1}{2uv} \right).
$$
Thus these equations define a branch of the affine mid-planes evolute. The implicit equation of this branch 
is $uvw=-\frac{2}{25}$. Observe that the tangent plane to this branch at $(u,v)=(1,1)$ is 
$$
5\left(u-\frac{2}{5} \right)+5\left(v-\frac{2}{5}\right)-4\left(w+\frac{1}{2}\right)=0,
$$
which coincides with the Transon plane of $S$ at $(p_0,T_1)$. Since in this example $b=\frac{\sqrt{2}}{6}$ is constant,
this result agrees with Lemma \ref{lemma:Transon1}.
\end{example}

\begin{lemma}\label{lemma:Transon2}
Assume that  the cubic form at $p$ vanishes at $(\xi,\eta)(0)$. 
Then necessarily $X'(0)$ belongs to the Transon plane. Moreover, 
$X'(0)$ is in the cone of B.Su if and only if $b'(0)=0$.
\end{lemma}

\begin{proof}
The first claim is a direct consequence of the proof of Lemma \ref{lemma:Transon1}. 
For the second claim, we may assume that $(\xi(0),\eta(0))=(1,0)$. For simplicity, we shall also assume that $9b^2-2f_{4,0}\neq 0$ so that
the center of the Moutard quadric is finite. Thus, at $t=0$, $X$ is given by equation \eqref{eq:CenterMoutard} with $f_{3,0}=0$, $f_{2,1}=-3b(0)$, i.e.,
$$
(x,y,z)(0)=-\frac{1}{4(9b^2-2f_{4,0})}\left(0, 6b(0),1 \right).
$$ 
Observe that
$$
\mathbf{G}_{\eta}(p,\xi,\eta)=\left(  \xi\eta,\frac{1}{2}(\xi^2+3\eta^2),3b(p)(\eta^2-\xi^2) \right).
$$
Differentiating $\mathbf{G}_{\eta}(t)\cdot X(t)=0$ we obtain
$$
\mathbf{G}_{\eta\xi}\cdot X\ \xi'+\mathbf{G}_{\eta\eta}\cdot X\ \eta'+\mathbf{G}_{\eta}\cdot X' +3b'(\eta^2-\xi^2)z=0.
$$
Now take $t=0$ and use $\xi(0)=1$, $\eta(0)=0$ to obtain
$$
\left(0,1,-6b(0)\right)\cdot X\ \xi'+\left( 1, 0, 0\right)\cdot X\ \eta'+\mathbf{G}_{\eta}\cdot (X')-3b'(0)z=0.
$$
Using the above formula for $(x,y,z)$ we get
$$
\mathbf{G}_{\eta}\cdot X'+\frac{3b'(0)}{4(9b^2-2f_{4,0})}=0,
$$
thus proving the lemma.
\end{proof}

\begin{lemma}\label{lemma:Su1}
Assume $p'(0)=T$ and denote by $\gamma$ the intersection of $S$ with the plane generated by $T$ and $s(T)$. If $\mu_\gamma'(0)\neq 0$, where $\mu_\gamma$ denotes the affine curvature of $\gamma$, $X'(0)$ is a non-zero vector in the direction of the cone of B.Su.
\end{lemma}

\begin{proof}
Since $X$ is a point of the affine evolute of $\gamma$ and $T$ is tangent to this curve,  $X'(0)$ is in the direction of the its affine normal, which belongs to the cone of B.Su
(see Proposition \ref{prop:centro_quadrica_moutard}).
\end{proof}

If we assume that $S$ is given by \eqref{eq:CubicaNormalizada}, $p=(0,0)$ and $T(p)=(1,0)$, the condition $\mu_\gamma'(0)\neq 0$ can be explicitly 
described. Expanding $f$ until order $5$, it is straightforward to verify that the projection of $\gamma$
in the $xz$ plane is given by
$$
z=\frac{1}{2}x^2+ax^3+\left(f_{4,0}-\frac{9}{2}b^2\right)x^4+\left(-27ab^2+3bf_{3,1}+f_{5,0}\right)x^5 +O(6).
$$
From appendix \ref{app1}, it follows that 
$$
\mu_\gamma'(0)=\frac{1}{27}\left( 9a_5+40a_3^3-45a_3a_4\right),
$$
where $a_3=6a$, $a_4=24\left(f_{4,0}-\frac{9}{2}b^2\right)$ and 
$a_5=120\left(-27ab^2+3bf_{3,1}+f_{5,0}\right)$.

\begin{proposition}
Assume that $p$ is not critical for the Pick invariant and that $\mu_\gamma'(0)\neq 0$.
Then the corresponding branch of the Affine Mid-Planes Evolute is smooth at $p$. 
\end{proposition}
\begin{proof}
From Lemma \ref{lemma:Su1}, $X_T$ is a non-zero vector in the direction of the cone of B.Su. Take a direction $W\neq T$
such that the derivative of $b$ in this direction in non-zero. From Lemmas \ref{lemma:Transon1} and \ref{lemma:Transon2}, $X_W$ is not a multiple of $X_T$,
thus proving the proposition.
\end{proof}


\bibliographystyle{amsplain}

\appendix

\section{Affine curvature of a planar curve}\label{app1}

In this appendix we prove the following result:
Let $\gamma$ be a planar curve such that $\gamma(0)=(0,0)$. Assume that, close to the origin, $\gamma$ is the graph of 
\begin{equation}
g(x)=\frac{1}{2}x^2+\frac{a_3}{6}x^3+\frac{a_4}{24}x^4+\frac{a_5}{120}x^5+O(6).
\end{equation}
Then the affine normal vector at the origin is given by $(-\frac{a_3}{3},1)$, the affine curvature at the origin is 
\begin{equation}\label{eq:FormulaMu}
\mu= \frac{1}{9}\left( 3a_4-5a_3^2 \right)
\end{equation}
and the derivative of the affine curvature with respect to affine arc-length is
\begin{equation}\label{eq:FormulaMuLinha}
\mu'= \frac{1}{27}\left( 9a_5+40a_3^3-45a_3a_4 \right).
\end{equation}

\begin{proof}
Differentiate  $\gamma(x)=\left(x,g(x)\right)$ two times with respect to the affine arc-length parameter to obtain $x_s=1$. Differentiate one more to obtain $x_{ss}=-\frac{a_3}{3}$, which implies 
first claim. Differentiating again we obtain 
$$
x_{sss}=-\frac{1}{3}a_4+\frac{5}{9}a_3^2,
$$
which together with 
\begin{equation}\label{eq:Mus}
\mu(s)=[\gamma''(s),\gamma'''(s)]
\end{equation}
implies equation \eqref{eq:FormulaMu}. Finally differentiate once more to obtain 
$$
x_{ssss}=-\frac{1}{3}a_5-\frac{5}{3}a_3^3+\frac{16}{9}a_3a_4.
$$
Differentiating equation \eqref{eq:Mus} and applying this formula we obtain formula \eqref{eq:FormulaMuLinha}.
\end{proof}

\section{Proofs of lemmas \ref{Lemma:Main3} and \ref{Lemma:Main4}}\label{app2}

Assume $f$ is given by equation \eqref{eq:CubicaNormalizada}. Then
\begin{equation*}
F=\left[(N_1\cdot C)N_2+(N_2\cdot C)N_1\right] \cdot (X-M),
\end{equation*}
where 
\begin{equation*}
C=\left( \Delta u, \Delta v, f_1-f_2\right),\ \ M=\frac{1}{2}\left( \sigma u, \sigma v, f_1+f_2\right),
\end{equation*}
and 
\begin{equation*}
N=\left( -(x+(f_3)_x), -(y+(f_3)_y),1 \right).
\end{equation*}

Thus
$$
(N_1\cdot C)N_2=\left[ -(u_1+(f_3)_x(1)) \Delta u- (v_1+(f_3)_y(1)) \Delta v+(f_1-f_2)\right] N_2,
$$
$$
(N_2\cdot C)N_1=\left[ -(u_2+(f_3)_x(2)) \Delta u- (v_2+(f_3)_y(2)) \Delta v+(f_1-f_2)\right] N_1,
$$
where $(f_3)_x(i)=(f_3)_x(u_i,v_i)$, and similarly for $(f_3)_y$.

To prove Lemma \ref{Lemma:Main3}, we must show that the expansion of $F$ at third order is $G(\Delta u,\Delta v,X)$.
We calculate the coefficient of order $3$ of $x$, the others being similar. Writing $\sigma u=u_1+u_2$, $\sigma v=v_1+v_2$, we can write 
$$
f_1-f_2=\frac{1}{2}(\Delta u\sigma u+\Delta v\sigma v)+O(3).
$$
Thus the $x$ coefficient of the expansion of $F$ at third order is
$$
-\left( u_1\Delta u+v_1\Delta v  \right)u_2-\left( u_2\Delta u+v_2\Delta v  \right)u_1+\frac{1}{2}(\Delta u\sigma u+\Delta v\sigma v)\sigma u=
$$
$$
=\Delta u\left(  -2u_1u_2+\frac{1}{2}(\sigma u)^2 \right)+\Delta v\left(  -v_1u_2-v_2u_1+\frac{1}{2}\sigma u\sigma v \right)
$$
$$
=\frac{1}{2}\Delta u^3+\frac{1}{2}(\Delta v)^2\Delta u.
$$
Thus, from formula \eqref{planodetranson}, we obtain that this coefficient is equal to the $x$ coefficient of $G(\Delta u,\Delta v,X)$.

\bigskip\bigskip

To prove Lemma \ref{Lemma:Main4}, we must show that the expansion of $F$ at order $4$ is 
$H_1(\Delta u,\Delta v)\sigma u+H_2(\Delta u,\Delta v)\sigma v$.
We shall calculate the coefficient of order $4$ of $x$, the others being similar. This coefficient is given by
$$
\left( -u_1\Delta u-v_1\Delta v +\frac{1}{2}(u_1^2-u_2^2+v_1^2-v_2^2) \right)(f_3)_x(2)
$$
$$
+\left( -u_2\Delta u-v_2\Delta v +\frac{1}{2}(u_1^2-u_2^2+v_1^2-v_2^2) \right)(f_3)_x(1)
$$
$$
+\left( -(f_3)_x(1)\Delta u-(f_3)_y(1)\Delta v +(f_3(1)-f_3(2)) \right)u_2
$$
$$
+\left( -(f_3)_x(2)\Delta u-(f_3)_y(2)\Delta v + (f_3(1)-f_3(2)) \right)u_1.
$$
Since $f_3$ is homogeneous of degree $3$, $(f_3)_xu_i+(f_3)_yv_i=3f_3(i)$. Thus this expression is equal to
$$
+\frac{1}{2}\left( (\Delta u)^2+(\Delta v)^2  \right) \left( (f_3)_x(1)- (f_3)_x(2) \right)
$$
$$
+\left( -2f_3(1)-f_3(2) +(f_3)_x(1)u_2+(f_3)_y(1)v_2\right)u_2
$$
$$
+\left( 2f_3(2)+ f_3(1)-(f_3)_x(2)u_1-(f_3)_y(2)v_1  \right)u_1.
$$
Consider first only terms in $a$. The first parcel leads to
$$
\frac{3a}{2}(\Delta u^2+\Delta v^2)(\Delta u\sigma u-\Delta v\sigma v).
$$
The latter two parcels lead to
$$
2au_1u_2(-u_1^2+3v_1^2+u_2^2-3v_2^2)+a(u_1^4-u_2^4)-3a(u_1^2+u_2^2)\Delta v\sigma v
$$
$$
=a\left[  \Delta u^3\sigma u-3\Delta v\Delta u^2\sigma v \right]
$$
Summing all we obtain 
$$
\left(\frac{5a}{2}\Delta u^3+\frac{3a}{2}\Delta u\Delta v^2\right)\sigma u-\left(\frac{9a}{2}\Delta u^2\Delta v+\frac{3a}{2}\Delta v^3\right)\sigma v.
$$
\noindent
Consider now terms in $b$:
$$
3(\Delta u^2+\Delta v^2)\left(u_2v_2-u_1v_1\right)
$$
$$
+\left[-2(v_1^3-3v_1u_1^2)-(v_2^3-3v_2u_2^2)-6u_1v_1u_2+3(v_1^2-u_1^2)v_2\right]u_2
$$
$$
+\left[2(v_2^3-3v_2u_2^2)+(v_1^3-3v_1u_1^2)+6u_2v_2u_1-3(v_2^2-u_2^2)v_1\right]u_1
$$
After some simplifications we obtain:
$$
6\Delta u^2(u_2v_2-u_1v_1)+6v_1v_2\Delta v\sigma u-2\Delta v\sigma u(v_1^2+v_1v_2+v_2^2)=
$$
$$
=6\Delta u^2(u_2v_2-u_1v_1)-2\Delta v^3\sigma u=-3\Delta u^3\sigma v-3\Delta u^2\Delta v\sigma u-2\Delta v^3\sigma u.
$$
Thus the coefficient of order $4$ of $x$ is 
$$
H_{11}(\Delta u, \Delta v)\sigma u+H_{21}(\Delta u,\Delta v)\sigma v,
$$
as we wish to prove.

\end{document}